\documentclass{amsart}

\usepackage{txfonts}
\newcommand{\Hom}{\normalfont\mbox{Hom}\,}
\newcommand{\Ext}{\normalfont\mbox{Ext}\,}
\newcommand{\Tor}{\normalfont\mbox{Tor}\,}
\newcommand{\pd}{\normalfont\mbox{pd}\,}
\newcommand{\id}{\normalfont\mbox{id}\,}
\newcommand{\fd}{\normalfont\mbox{fd}\,}
\newcommand{\gd}{\mbox{Gdim}\,}
\newcommand{\wdim}{\normalfont\mbox{wdim}\,}
\newcommand{\FP}{\normalfont\mbox{FP}\,}
\newcommand{\IF}{\normalfont\mbox{IF}\,}
\newcommand{\IFD}{\normalfont\mbox{IFD}\,}
\newcommand{\gldim}{\normalfont\mbox{gldim}\,}
\newcommand{\wGgldim}{\normalfont\mbox{wGgldim}\,}
\newcommand{\Ggldim}{\normalfont\mbox{Ggldim}\,}
\newcommand{\Gpd}{\normalfont\mbox{Gpd}\,}
\newcommand{\Gid}{\normalfont\mbox{Gid}\,}
\newcommand{\Gfd}{\normalfont\mbox{Gfd}\,}
\newcommand{\im}{\normalfont\mbox{Im}\,}

\theoremstyle{plain}
\newtheorem{theorem}{Theorem}[section]
\newtheorem{lemma}[theorem]{Lemma}
\newtheorem{proposition}[theorem]{Proposition}
\newtheorem{corollary}[theorem]{Corollary}

\theoremstyle{definition}

\newtheorem{remark}[theorem]{Remark}
\newtheorem{example}[theorem]{Example}

\theoremstyle{Definition and Notation}

\catcode`\ç=13
\defç{\c{c}}
\catcode`\é=13
\defé{\'e}
\catcode`\à=13
\defà{\`a}
\catcode`\è=13
\defè{\`e}
\catcode`\â=13
\defâ{\^a}
\catcode`\ù=13
\defù{\`u}
\catcode`\ê=13
\defê{\^e}
\catcode`\î=13
\defî{\^\i}
\catcode`\ô=13
\defô{\^o}
\begin{document}

\title[]{Note on (weak) Gorenstein global dimensions}

\author{Najib Mahdou}
\address{Najib Mahdou\\Department of Mathematics, Faculty of Science and Technology of
Fez, Box 2202, University S.M. Ben Abdellah Fez, Morocco. \\
mahdou@hotmail.com }

\author{Mohammed Tamekkante}
\address{Mohammed Tamekkante\\Department of Mathematics, Faculty of Science and Technology
of Fez, Box 2202, University S.M. Ben Abdellah Fez, Morocco. \\
tamekkante@yahoo.fr }



\begin{abstract}
In this note we characterize the (resp., weak) Gorenstein global dimension for an
arbitrary ring. Also, we extend the well-known Hilbert's syzygy Theorem to the weak Gorenstein global dimension and  we study the weak Gorenstein
homological dimensions of direct product of rings, which gives examples
of non-coherent  rings of finite Gorenstein dimensions $>0$ and infinite classical
weak dimension.
\end{abstract}
\maketitle
\section{Introduction}

Throughout this paper, $R$ denotes -if
not specified otherwise- a non-trivial associative ring  and all
modules are unitary.\\
Let $R$ be a ring, and let $M$ be an $R$-module. As usual we use
$\pd_R(M)$, $\id_R(M)$ and $\fd_R(M)$ to denote, respectively, the
classical projective dimension, injective dimension and flat
dimension of $M$. We use also $\gldim(R)$ and $\wdim(R)$ to denote, respectively, the classical global and weak dimension of $R$.\\

For a two-sided Noetherian ring $R$, Auslander and Bridger
\cite{Aus bri} introduced the $G$-dimension, $\gd_R (M)$, for
every finitely generated $R$-module $M$. They showed that there is
an inequality $\gd_R (M)\leq \pd_R (M)$ for all finite $R$-modules
$M$, and equality holds if $\pd_R (M)$ is finite.

Several decades later, Enochs and Jenda \cite{Enochs,Enochs2}
defined the notion of Gorenstein projective dimension
($G$-projective dimension for short), as an extension of
$G$-dimension to modules which are not necessarily finitely
generated, and the Gorenstein injective dimension ($G$-injective
dimension for short) as a dual notion of Gorenstein projective
dimension. Then, to complete the analogy with the classical
homological dimension, Enochs, Jenda and Torrecillas \cite{Eno
Jenda Torrecillas} introduced the Gorenstein flat dimension. Some
references are
 \cite{Christensen, Christensen
and Frankild, Enochs, Enochs2, Eno Jenda Torrecillas, Holm}.\\

Recall that a left (resp., right) $R$-module $M$ is called
Gorenstein projective if, there exists an exact sequence of
projective left (resp., right) $R$-modules:
$$\mathbf{P}:...\rightarrow P_1\rightarrow P_0\rightarrow
P^0\rightarrow P^1\rightarrow ...$$ such that $M\cong
\im(P_0\rightarrow P^0)$ and such that the operator $\Hom_R(-,Q)$
leaves $\mathbf{P}$ exact whenever $Q$ is a left (resp., right)
projective $R$-module. The resolution $\mathbf{P}$ is called a complete projective resolution. \\
The left and right Gorenstein injective $R$-module are defined
dually.\\
 And an $R$-module $M$ is called left (resp., right)
Gorenstein flat if, there exists an exact sequence of flat left
(resp., right) $R$-modules:
$$\mathbf{F}:...\rightarrow F_1\rightarrow F_0\rightarrow
F^0\rightarrow F^1\rightarrow ...$$ such that $M\cong
\im(P_0\rightarrow P^0)$ and such that the operator $I\otimes_R-$
(resp., $-\otimes_RI$) leaves $F$ exact whenever $I$ is a right
(resp., left) injective $R$-module. The resolution $\mathbf{F}$ is called complete flat resolution.\\
Using the definitions, we immediately get the following
characterization of Gorenstein flat modules:
\begin{lemma}\label{lemma}
An $R$-module $M$ is Gorenstein left (resp., right) flat if, and only if,
\begin{enumerate}
  \item $\Tor^i_R(I,M)=0$ (resp.$\Tor^i_R(M,I)=0$), for every right (resp., left) injective $R$-module
$I$ and all $i>0$.
  \item There exists an exact sequence of left (resp, right) $R$-modules $0\rightarrow M \rightarrow F_0\rightarrow F_1\rightarrow...$ where each $F_i$ is  flat such that
the functor $I\otimes_R-$ (resp., $-\otimes_RI$) keeps the
exactness of this sequence whenever $I$ is right injective.
\end{enumerate}
\end{lemma}
\begin{remark}\label{remark}
Using the Lemma above and an $n$-step flat resolution of left (resp., right) $R$-module $M$ we conclude that if $\Gfd_R(M)\leq n$ then $\Tor_R^i(I,M)=0$ (resp., $\Tor^i_R(M,I)=0$) for every right (resp., left) injective $R$-module $I$ and all $i>n$. The inverse implication is given by Holm (\cite[Theorem 3.14]{Holm}) when $\Gfd_R(M)<\infty$ and the ring is left (resp., right)   coherent.
\end{remark}
The  Gorenstein projective, injective and flat
dimensions are defined in term of resolution and denoted by $\Gpd(-)$, $\Gid(-)$ and $\Gfd(-)$ respectively (see \cite{Christensen, Enocks and janda, Holm}).\\

In \cite{Bennis and Mahdou2}, the authors prove the equality:
$$\sup\{\Gpd_R(M)\mid \text{M is a left  R-module}\}=\sup\{\Gid_R(M)\mid \text{M is a left  R-module}\}$$
They called the common value of the above quantities the left
Gorenstein global dimension of $R$ and denoted it by
$l.\Ggldim(R)$. Similarly, they set
$$l.\wGgldim(R)=\sup\{\Gfd_R(M)\mid \text{M is a left  R-module}\}$$
which they called the left weak Gorenstein global dimension of
$R$. Similarly with the right modules, we can define the right
Gorenstein global and weak dimensions; $r.\Ggldim(R)$ and $r.\wGgldim(R)$. When $R$ is a commutative ring, we drop the unneeded letters $r$ and $l$.\\
The Gorenstein global dimension measures how far away a ring $R$
is from being quasi-Frobenius (i.e; Noetherian and self injective
rings) (see \cite[Proposition 2.6]{Bennis and Mahdou2}). On the
other hand, from Faith-Walker Theorem \cite[Theorem
7.56]{Nicolson}, a ring is quasi-Frobenius if, and only if, every
injective right (resp., left) module is projective or equivalently
every projective right (resp., left) module is injective. Hence,
from \cite[Proposition 2.6]{Bennis and Mahdou2}, we have the
following Corollary:
 \begin{corollary}\label{Ggd(R)=0}
The following statements are equivalent:
\begin{enumerate}
  \item $l.\Ggldim(R)=0$.
  \item $r.\Ggldim(R)=0$.
  \item Every left (and right) projective  $R$-module is injective.
\end{enumerate}
\end{corollary}
For rings with high $l.\Ggldim(-)$, \cite[Lemma 2.1]{Bennis and Mahdou2} gives a nice characterization to $l.\Ggldim(R)$ for an arbitrary ring $R$ provided the finiteness of this dimension as  shown by the next Proposition:
\begin{proposition}[Lemma 2.1, \cite{Bennis and Mahdou2}]\label{lemma bennisMahdou}
If $l.\Ggldim(R)<\infty$, then the following statements are
equivalent:
\begin{enumerate}
  \item $l.\Ggldim(R)\leq n$.
  \item $\id_R(P)\leq  n$ for every left $R$-module $P$ with finite projective dimension.
\end{enumerate}
\end{proposition}
There is a similar result of Corollary \ref{Ggd(R)=0} for the weak Gorenstein global dimension as shown by the bellow  Proposition. Recall that a ring is called  right (resp., left)  IF ring if   every right (resp., left) injective module is flat and it is called IF ring if it is both right and left IF ring (see please \cite{Ding}).
\begin{proposition}\label{lemma0}
The following statements are equivalent for every  ring $R$:
\begin{enumerate}
  \item $l.\wGgldim(R)=0$
  \item Every left and every right injective $R$-module is flat (i.e, $\IF$ ring).
  \item $r.\wGgldim(R)=0$
\end{enumerate}
\end{proposition}
\begin{proof} We prove the implications $(1\Rightarrow 2 \Rightarrow 3)$ and the inverse implications are similar.\\
$(1\Rightarrow 2$). Suppose that $l.\wGgldim(R)=0$. Let $I$ be right injective $R$-module. For an arbitrary left $R$-module $M$ and all
$i>0$ we have $\Tor_R^i(I,M)=0$ (see Lemma \ref{lemma}). Then, $I$ is flat. Moreover, since every left $R$-module is
Gorenstein flat (since $l.\wGgldim(R)=0$), every left $R$-module can be embedding  in a left flat $R$-module. In particular,
every left injective $R$-module is contained in a flat module. Then, every left injective $R$-module is a direct summand
of a flat module and so it is flat, as desired.\\
$(2\Rightarrow 3).$ Let $M$ be a right $R$-module. Assemble any
flat resolution  of $M$  with its any injective resolution, we get
an exact sequence of right flat   $R$-modules $\mathbf{F}$ (since
every right  injective module is flat). Also, since every left
injective module $I$ is flat, $\mathbf{F}\otimes_RI$ is exact and
so $\mathbf{F}$ is a complete flat resolution. This means that $M$
is Gorenstein flat. Consequently, $r.\wGgldim(R)=0$.
\end{proof}

The aim of this note is to give generalizations of Corollary
\ref{Ggd(R)=0} and Proposition \ref{lemma0} in the way of
Proposition \ref{lemma bennisMahdou} for an arbitrary ring with
high (weak) Gorenstein global dimension (see Theorem \ref{main
result1}, \ref{main result2'} and \ref{main result2}). Also, we
extend the well-known Hilbert's syzygy Theorem to the weak
Gorenstein global dimension and  we study the weak Gorenstein
homological dimension of direct product of rings\footnote{The
extension of the Hilbert's syzygy Theorem  and  the study the weak
Gorenstein homological dimensions of direct product of rings to
weak Gorenstein dimension was done in \cite{Bennis and Mahdou2}
over a coherent rings. There we  give a generalization to an
arbitrary ring.}, which gives examples of non-coherent  rings of
finite Gorenstein dimensions $>0$ and infinite classical
weak dimension.\\
Throughout the rest of this paper, all modules are-if not
specified otherwise- left $R$-modules. The definitions and
notations employed in this paper are based on those introduced by
Holm in \cite{Holm}.

\section{main results}
Our first main result of this paper is the following Theorem:

\begin{theorem}\label{main result1}
Let $R$ be a ring and $n$ a positive integer. Then, $l.\Ggldim(R)\leq n$ if, and only if, $R$  satisfies the following two conditions:
\begin{description}
  \item[(C1)] $\id_R(P)\leq n$ for every projective (left) $R$-module.
  \item[(C2)] $\pd_R(I)\leq n$ for every injective (left) $R$-module.
\end{description}
\end{theorem}

\begin{proof}
$(\Rightarrow).$ Suppose that $l.\Ggldim(R)\leq n$. We claim $(\mathbf{C1})$. Let $P$ be a projective  $R$-module. Since $\Gpd_R(M)\leq n$ for every  $R$-module $M$, we have $\Ext_R^i(M,P)=0$ for all $i>n$ (by \cite[Theorem 2.20]{Holm}). Hence, $\id_R(P)\leq n$, as desired.\\
Now, we claim $(\mathbf{C2})$. Let $I$ be an injective $R$-module. Since $l.\Ggldim(R)=\sup\{\Gid_R(M)\mid\text{M is a left  R-module}\}$, for an arbitrary  $R$-module $M$ we have $\Ext_R^i(I,M)=0$ for all $i>n$ (by \cite[Theorem 2.22]{Holm}). Hence, $\pd_R(I)\leq n$, as desired.\\
$(\Leftarrow).$ Suppose that $R$ satisfies $(\mathbf{C1})$ and
$(\mathbf{C2})$ and we claim that $l.\Ggldim(R)\leq n$. Let $M$ be
an arbitrary  $R$ module and consider an $n$-step  projective
resolution of $M$ as follows:
$$0\rightarrow G\rightarrow P_n\rightarrow ....\rightarrow P_1\rightarrow M \rightarrow 0$$ where all $P_i$ are projective. We have to prove that $G$ is
Gorenstein projective. First, for every  projective $R$-module $P$
and  all $i>0$, we have  $\Ext_R^i(G,P)=\Ext_R^{n+i}(M,P)=0$ by
condition $(\mathbf{C1})$. So, from \cite[Proposition 2.3]{Holm},
it suffices to prove that $G$ admits a right co-proper projective
resolution (see \cite[Definition 1.5]{Holm}). Pick a short exact
sequence $0\rightarrow M \rightarrow I\rightarrow M'\rightarrow 0$
where $I$ is an  injective $R$-module and for $M'$ consider an
$n$-step projective resolution as follows:
 $$0\rightarrow G'\rightarrow P_n'\rightarrow ....\rightarrow P_1'\rightarrow M' \rightarrow 0$$
 We have the following diagram:
 $$\begin{array}{ccccccccc}
 & 0 &  &  0 & & 0 &  & 0 &  \\
    & \downarrow & & \downarrow &  &\downarrow &  & \downarrow &  \\
   0\rightarrow & G & \rightarrow & P_n & \rightarrow ...\rightarrow  & P_1 & \rightarrow & M  & \rightarrow 0 \\
   & \downarrow & & \downarrow &  &\downarrow &  & \downarrow &  \\
   0\rightarrow & Q_1 & \rightarrow & P_n\oplus P_n' & \rightarrow ...\rightarrow  & P_1\oplus P_1' & \rightarrow & I  & \rightarrow 0 \\
   & \downarrow & & \downarrow &  &\downarrow &  & \downarrow &  \\
   0\rightarrow & G' & \rightarrow & P_n' & \rightarrow ...\rightarrow  & P_1' & \rightarrow & M'  & \rightarrow 0 \\
   & \downarrow & & \downarrow &  &\downarrow &  & \downarrow &  \\
 & 0 &  &  0 & & 0 &  & 0 &
 \end{array}$$
 Since $\pd_R(I)\leq n$ (by $(\mathbf{C2})$), the module $Q_1$ is clearly projective. On the other hand, we have
 $\Ext_R(G',P)=\Ext_R^{n+1}(M',P)=0$ for every projective module $P$ by $(\mathbf{C1})$. Thus, the functor $\Hom_R(-,P)$
 keeps
 the exactness of the short exact sequence $0\rightarrow G \rightarrow Q_1 \rightarrow G' \rightarrow 0$.
 By repeating this procedure we obtain a right projective resolution $$0\rightarrow G \rightarrow Q_1 \rightarrow Q_2 \rightarrow ...$$
 such that $\Hom_R(-,P)$ leaves this sequence exact whenever $P$ is projective. Hence, $G$ is Gorenstein projective.
 Consequently, $\Gpd_R(M)\leq n$ and so $l.\Ggldim(R)\leq n$, as desired.
\end{proof}
If we denote $l.\mathcal{P}(R)$ (resp., $r.\mathcal{P}(R)$) and
$l.\mathcal{I}(R)$ (resp., $r.\mathcal{I}(R)$) , respectively, the
set of all left (resp., right)  projective and injective
$R$-modules, we have:
\begin{eqnarray*}
l.\Ggldim(R) &=& \sup\{\pd_R(I),\id_R(P) \mid I\in l.\mathcal{I}(R), P\in l.\mathcal{P}(R)\} \; \text{and}  \\
  r.\Ggldim(R) &=&  \sup\{\pd_R(I),\id_R(P) \mid I\in r.\mathcal{I}(R), P\in r.\mathcal{P}(R)\}.
\end{eqnarray*}
 There is another way to write the above Theorem:
 \begin{corollary}
 Let $R$ be a ring and $n$ be a positive integer. The following statements are equivalent:
 \begin{enumerate}
   \item $l.\Ggldim(R)\leq n$.
   \item For any $R$-module $M$: $\pd_R(M)\leq n\Leftrightarrow \id_R(M)\leq n$.
 \end{enumerate}
 \end{corollary}
 \begin{proof}
$(1\Rightarrow 2).$ Let $M$ be an $R$-module such that $\pd_R(M)\leq n$. For such module, consider a projective resolution as follows:
$$0\rightarrow P_n\rightarrow ...\rightarrow P_1\rightarrow P_0\rightarrow M\rightarrow 0$$
From Theorem \ref{main result1}, $\id_R(P_i) \leq n$ for each $i=0,..,n$. Hence,  $\id_R(M)\leq n$. Similarly we prove that $\pd_R(M)\leq n$ for every $R$-module such that $\id_R(M)\leq n$.\\
$(2\Rightarrow 1).$ Follows directly from Theorem \ref{main
result1} since  the conditions $\mathbf{C1}$ and $\mathbf{C2}$ are
clearly satisfied.
\end{proof}

 \begin{proposition}
Let $R$ be a ring with finite Gorenstein global dimension. Then,
$(\mathbf{C1})$ and $(\mathbf{C2})$ of Theorem  \ref{main result1}
are equivalent and so the following statements are equivalent:
\begin{enumerate}
  \item $l.\Ggldim(R)\leq n$.
  \item $\id_R(P)\leq n$ for every  projective $R$-module.
  \item $\pd_R(I)\leq n$ for every  injective $R$-module.
\end{enumerate}
\end{proposition}

\begin{proof}
From Theorem \ref{main result1}, only the equivalence of
$(\mathbf{C1})$ and $(\mathbf{C2})$ need a proof. So, we prove
$(\mathbf{C1}\Rightarrow \mathbf{C2})$ and the other implication
is analogous.  Let $M$ be an arbitrary left $R$-module. For every
projective $R$-module $P$ and  all $i>n$, we have
$\Ext_R^i(M,P)=0$ (by $(\mathbf{C1})$). Then, from \cite[Theorem
2.20]{Holm}, $\Gpd_R(M)\leq n$. Hence, we have $l.\Ggldim(R)\leq
n$. So, by Theorem \ref{main result1}, $(\mathbf{C2})$ is
satisfied, as desired.
\end{proof}
 Our second main result of this paper is given by the bellow Theorem. Recall that over a  ring $R$,  Ding (\cite{Ding}) defined and investigated two global dimensions as follows:
 \begin{eqnarray*}
   r.\IFD(R) &=& \sup\{\fd_R(I)\mid \text{I is a right injective R-module}\}  \\
   l.\IFD(R) &=& \sup\{\fd_R(I)\mid \text{I is a left injective R-module}\}
 \end{eqnarray*}
For such dimensions, in \cite{Ding}, Ding gave a several
characterizations over an arbitrary ring and also over a coherent
ring. Recall also that a right (resp., left)  $R$-module $M$  is
called $\FP$-injective (or absolutely pure) if $\Ext_R(N,M)=0$ (or
equivalently $\Ext_R^i(N,M)=0$ for all $i>0$) for every finitely
presented right (resp. left)  $R$-module $N$. The $\FP$-injective
dimension of right (resp., left) $M$, denoted $\FP-\id_R(M)$,  is
defined to be the lest nonnegative integer $n$ such that
$\Ext_R^{n+1}(N,M)=0$ for every finitely presented right (resp.,
left) $R$-module (see \cite{Ding, Stenstrom}).
 \begin{theorem}\label{main result2'}
Let $R$ be a ring and $n$ a positive integer. The following
conditions are equivalent:
\begin{enumerate}
  \item $\sup\{l.\wGgldim(R),r.\wGgldim(R)\}\leq n$.
  \item $\Gfd_R(R/I)\leq n$ for every left and every right  ideal $I$.
  \item $\fd_R(E)\leq n$ for every left and every right injective $R$-module $E$.
\end{enumerate}
Consequently:
\begin{eqnarray*}
  \sup\{l.\wGgldim(R),r.\wGgldim(R)\} &=& \sup\{\fd_R(I)\mid I\in l.\mathcal{I}(R)\cup r.\mathcal{I}(R)\} \\
   &=&   \sup\{l.\IFD(R),r.\IFD(R)\}
\end{eqnarray*}
\end{theorem}
\begin{proof}
$(1\Rightarrow 2).$  Obvious by  definition of the left and right weak Gorenstein global dimension.\\
$(2\Rightarrow 3).$ Let $E$ be a left  injective $R$-module. Since $\Gfd_R(R/I)\leq n$ for every right  ideal $I$ and  from Remark \ref{remark}, we get $\Tor^i_R(R/I,E)=0$ for all $i>n$. Hence, from \cite[Lemma 9.18]{Rotman}, $\fd_R(E)\leq n$. Similarly we prove that $\fd_R(E)\leq n$ for every right injective $R$-module. \\
$(3\Rightarrow 1).$ Let $M$ be an arbitrary  left $R$-module and
consider an $n$-step projective resolution of $M$ as follows:
$$0\rightarrow G\rightarrow P_n\rightarrow ....\rightarrow P_1\rightarrow M \rightarrow 0$$ where all $P_i$ are left projective.
We have to prove that $G$ is a Gorenstein flat $R$-module. First,
for every  right injective  $R$-module $E$ we have
$\Tor_R^i(E,G)=\Tor_R^{n+i}(E,M)=0$ for all $i>0$ since
$\fd_R(E)\leq n$ (by hypothesis).\\ Now,  Pick a short exact
sequence of left $R$-modules $0\rightarrow M \rightarrow
E_1\rightarrow M'\rightarrow 0$ where $E_1$ is an  injective
$R$-module, and for $M'$ consider an $n$-step projective
resolution as follows:
 $$0\rightarrow G'\rightarrow P_n'\rightarrow ....\rightarrow P_1'\rightarrow M' \rightarrow 0$$
 We have the following diagram:
 $$\begin{array}{ccccccccc}
 & 0 &  &  0 & & 0 &  & 0 &  \\
    & \downarrow & & \downarrow &  &\downarrow &  & \downarrow &  \\
   0\rightarrow & G & \rightarrow & P_n & \rightarrow ...\rightarrow  & P_1 & \rightarrow & M  & \rightarrow 0 \\
   & \downarrow & & \downarrow &  &\downarrow &  & \downarrow &  \\
   0\rightarrow & F_1 & \rightarrow & P_n\oplus P_n' & \rightarrow ...\rightarrow  & P_1\oplus P_1' & \rightarrow & E_1  & \rightarrow 0 \\
   & \downarrow & & \downarrow &  &\downarrow &  & \downarrow &  \\
   0\rightarrow & G' & \rightarrow & P_n' & \rightarrow ...\rightarrow  & P_1' & \rightarrow & M'  & \rightarrow 0 \\
   & \downarrow & & \downarrow &  &\downarrow &  & \downarrow &  \\
 & 0 &  &  0 & & 0 &  & 0 &
 \end{array}$$
 Since $\fd_R(E_1)\leq n$, the module $F_1$ is clearly left flat. On the other hand, we have
 $\Tor_R^1(E,G')=\Tor_R^{n+1}(E,M')=0$ for every right injective  $R$-module $E$ (since $\fd_R(E)\leq n$).
 Thus, the functor $E\otimes_R-$  keeps the exactness of the short exact sequence $0\rightarrow G \rightarrow F_1 \rightarrow G' \rightarrow 0$.
 By repeating this procedure we obtain a  flat resolution of $G$ as follows: $$0\rightarrow G \rightarrow F_1 \rightarrow F_2 \rightarrow ...$$ such that $E\otimes_R-$ leaves this sequence exact whenever $E$ is right injective. Hence, from Lemma \ref{lemma}, $G$ is left Gorenstein flat.  Then, $\Gfd_R(M)\leq n$. Consequently, $l.\wGgldim(R)\leq n$, as desired.\\
 Similarly, we prove that $r.\wGgldim(R)\leq n$.
\end{proof}
 It is true that $l.\wGgldim(R)\leq n$ implies that $\fd_R(I)\leq n$ for every right injective $R$-module (by Remark \ref{remark}).
 But the inverse implication is not true in the general case as shown in the next Example. Thus explicates the form of Theorem  \ref{main result2'}.
\begin{example}\label{example0}
Let $R$ be a left and right coherent ring $R$ which is right IF but not left IF (see \cite[Example 2]{colby}). Then, $l.\wGgldim(R)=r.\wGgldim(R)=\infty$.
\end{example}

\begin{proof}
If $l.\wGgldim(R)<\infty$, then using \cite[Theorem 3.14]{Holm} and since every right injective $R$-module is flat (since $R$ is right IF ring), we have $\Gfd_R(M)=0$ for every left $R$-module $M$ and so $l.\wGgldim(R)=0$. So from Proposition  \ref{lemma0} every left injective module is flat. But, this contradicts the fact that $R$ is not left IF.
 Now, if $r.\wGgldim(R)=n<\infty$. Then, $\fd_R(E)\leq n$ for every left injective $R$-module $E$. On the other hand,  $\fd_R(E')=0\leq n$ for every right injective $R$-module $E'$ since $R$ is right $IF$. Then, from Theorem \ref{main result2'}, $\sup\{l.\wGgldim(R),r.\wGgldim(R)\}\leq n$. Absurd, since $l.\wGgldim(R)=\infty$.
\end{proof}

Over a right  coherent ring, the characterization of
$l.\wGgldim(R)$ is more simple as shown in the next Proposition:
\begin{proposition}\label{propo right coherent}
Let $R$ be a right coherent ring. Then, $$l.\wGgldim(R)=\sup\{l.\IFD(R),r.\IFD(R)\}$$

\end{proposition}

\begin{proof}
From Theorem \ref{main result2'}$(3\Rightarrow 1)$, only the
inequality  $(\geq )$ need a proof. So, assume that
$l.\wGgldim(R)\leq n<\infty$. Clearly $l.\wGgldim(R)\geq
r.\IFD(R)$ since $\fd_R(I)\leq n$ for every right injective module
$I$ (by Remark \ref{remark}). So, we have to prove this fact for
$l.\IFD(R)$. Let $E$ be a  left injective $R$-module . Since
$l.\wGgldim(R)\leq n$, we have $\Gfd_R(E)\leq n$. Then, from
\cite[Lemma 3.17]{Holm}, there exists a short exact sequence
$0\rightarrow K \rightarrow G \rightarrow E \rightarrow 0$ where
$G$ is left Gorenstein flat and $\fd_R(K)\leq n-1$ (if $n=0$, this
should be interpreted as $K = 0$). Pick a short exact sequence
$0\rightarrow G\rightarrow F \rightarrow G' \rightarrow 0$ where
$F$ is left flat and $G'$ is left Gorenstein flat. Hence, consider
the following pushout diagram:
$$\begin{array}{ccccccc}
   & 0 &  & 0 &  &  &  \\
   & \downarrow &  & \downarrow &  & &  \\
   & K& =& K &  &  &  \\
   & \downarrow &  & \downarrow &  &  &  \\
  0\rightarrow & G & \rightarrow & F  & \rightarrow & G' & \rightarrow 0 \\
   & \downarrow &  & \downarrow &  & \parallel &  \\
  0\rightarrow & E & \rightarrow & D & \rightarrow  & G' & \rightarrow 0 \\
 & \downarrow &  & \downarrow &  & &  \\
 & 0 &  & 0 &  &  &  \\
\end{array}$$
Clearly, $\fd_R(D)\leq n$ and $E$ is containing in $D$. So, it is a direct summand of $D$ since it is injective. Therefore, $\fd_R(E)\leq n$. Consequently, $l.\wGgldim(R)\geq l.\IFD(R)$,  as desired.
\end{proof}

Similarly, we have:
\begin{proposition}\label{propo left coherent}
Let $R$ be a left coherent ring. Then, $$r.\wGgldim(R)=\sup\{l.\IFD(R),r.\IFD(R)\}$$
\end{proposition}

\begin{corollary}\label{coro}
Let $R$ be a ring. The following statements hold:
\begin{enumerate}
  \item If $R$ is right coherent, then $r.\wGgldim(R)\leq l.\wGgldim(R)$.
  \item If $R$ is left coherent, then $l.\wGgldim(R)\leq r.\wGgldim(R)$.
\end{enumerate}
Consequently, if $R$ is two-sided coherent, $r.\wGgldim(R) = l.\wGgldim(R)$

\end{corollary}

\begin{proof}
We suggest to prove (1) and the proof of $(2)$ will be similar. If
$R$ is right coherent, we have:
\begin{eqnarray*}
  l.\wGgldim(R) &=& \sup\{l.\IFD(R),r.\IFD(R)\} \qquad  \text{(from Proposition \ref{propo right coherent}).} \\
   & = & \sup\{r.\wGgldim(R),l.\wGgldim(R)\} \quad  \text{(from Theorem \ref{main result2'}).}
\end{eqnarray*}
So, we obtain the desired result.
\end{proof}

\begin{remark}
Using Theorem \ref{main result2'}, Propositions \ref{propo left
coherent}, Proposition \ref{propo right coherent} and Corollary
\ref{coro}, we can find many other characterizations of
$l.\wGgldim(R)$ and $r.\wGgldim(r)$ by using the characterizations
of the $l.\IFD(R)$ and $r.\IFD(R)$. For example we use,
\cite[Theorems 3.5 and  3.8, Proposition 3.17, Corollary 3.18
]{Ding}.
\end{remark}

A commutative version of Theorem \ref{main result2'} is as
follows.
\begin{theorem}\label{main result2}
Let $R$ be a commutative ring and $n$ be a positive integer. The
following conditions are equivalent:
\begin{enumerate}
   \item $\wGgldim(R)\leq n$
   \item $\Gfd_R(R/I)\leq n$ for every  ideal $I$.
   \item $\fd_R(E)\leq n$ for every injective $R$-module $E$.
 \end{enumerate}
 Consequently, $\wGgldim(R)=\IFD(R)$.
 \end{theorem}

\begin{proposition}[Theorems 3.5, 3.8 and  3.21, \cite{Ding}]
For any commutative ring the following conditions are equivalent:
\begin{enumerate}
  \item $\wGgldim(R)(=\IFD(R))\leq n$.
  \item $\fd_R(M)\leq n$ for every $\FP$-injective module $M$.
  \item $\fd_R(M)\leq n$ for every $R$-module $M$ with $\FP$-$id_R(M)<\infty$.
  \item $\id_R(\Hom_R(A,B))\leq n$ for every $\FP$-injective module $A$ and for every injective
module $B$.
  \item $\fd_R(\Hom_R(F,B))\leq n$ for every flat modules $F$ and all injective module $B$.
\end{enumerate}
Moreover, if $R$ is coherent, $\wGgldim(R)=\FP-\id_R(R)$.
\end{proposition}

In , \cite[Theorem 2.11 and 3.5]{Bennis and Mahdou3} the authors
prove that:
 \begin{description}
   \item[R1] If $\{R_i\}_{i=1}^n$ is a family of coherent commutative rings then:
    $$\wGgldim( \displaystyle \prod_{i=1}^{n}R_i)=sup\{\wGgldim(R_i); 1\leq i\leq n\}.$$
   \item[R2] If the polynomial ring $R[x]$ in one indeterminate $x$ over a commutative ring $R$ is
coherent, then: $\wGgldim(R[x]) = \wGgldim(R) + 1$.
 \end{description}

  In the next Theorems we see that the coherence condition is not necessary in $\mathbf{R1}$ and $\mathbf{R2}$.

\begin{theorem}\label{direct product}
For every family of commutative rings $\{R_i\}_{i=1}^n$ we have:
\begin{center}
$\wGgldim(\displaystyle\prod_{i=1}^nR_i)=\sup\{wGgldim(R_i),\;  1\leq i\leq n\}.$
\end{center}
\end{theorem}

\begin{proof}
By induction on $n$ it suffices to prove this result for $n=2$. \\
Assume that $\wGgldim(R_1\times R_2)\leq k$. Let $M_i$ be an $R_i$-module for $i=1,2$. Since each $R_i$ is  projective $R_1\times R_2$-module, by \cite[Proposition 3.10]{Holm} we have $\Gfd_{R_i}(M_i)\leq \Gfd_{R_i\times R_2}(M_1\times M_2)\leq k$. This follows that $\wGgldim(R_i)\leq k$ for each $i=1,2$.\\
Conversely, suppose that $\sup\{\wGgldim(R_i),\; 1=1,2\}\leq k$.
Let $I$ be an arbitrary injective $R_1\times R_2$-module. We can
see that $I\cong \Hom_{R_1\times R_2}(R_1\times R_2,I)\cong
\Hom_{R_1\times R_2}(R_1,I)\times \Hom_{R_1\times R_2}(R_2,I)$ and
that $I_i=\Hom_{R_1\times R_2}(R_i,I)$ is an  injective
$R_i$-module for each $i=1,2$. Since $\wGgldim(R_i)\leq k$ for
each $i$, we get that $\fd_{R_i}(I_i)\leq k$ (by Theorem \ref{main
result2}). Using \cite[Lemma 3.7]{Bennis and Mahdou3}, we have
$\fd_{R_1\times R_2}(I_1\times I_2)=\sup\{\fd_{R_i}(I_1),1\leq
i\leq 2\}\leq k$. Consequently, by Theorem \ref{main result2},
$\wGgldim(R_1\times R_2)\leq k$. This completes the proof.
\end{proof}

\begin{theorem}
Let $R[x]$ be  the polynomial ring  in one indeterminate $x$ over a commutative ring $R$. Then: $\wGgldim(R[x]) = \wGgldim(R) + 1$.
\end{theorem}
To prove this Theorem  we need the following Lemmas. Note that $M^+$ denote the character $\Hom_{\mathbb{Z}}(M,\mathbb{Q}/\mathbb{Z})$ of $M$.

\begin{lemma} \cite[Theorem 2.1]{character2} \label{lemma1}
Let $R$ be any ring and $M$ an $R$-module. Then,
$\fd_R(M)=\id_R(M^+)$.
\end{lemma}

\begin{lemma}  \cite[Theorem 202]{Kaplansky} \label{lemma2}
Let $R$ be any ring (not necessarily commutative). Let $x$ be a
central non-zero-divisor in $R$, and write $R^{\ast} = R/(x)$. Let
$A$ be a non-zero $R^{\ast}$-module with $\id_{R^{\ast}}(A) = n <
\infty$. Then $\id_R(A) = n + 1$.
\end{lemma}
\begin{proof}
First, we will prove that $\wGgldim(R)\leq \wGgldim(R[x])$. Let
$I$ be an arbitrary injective $R$-module. Clearly, the
$R[X]$-module $\Hom_R(R[x],I)$ is injective. Hence, from Theorem
\ref{main result2}, $\fd_{R[x]}(\Hom_R(R[x],I))\leq
\wGgldim(R[x])$. On the other hand, from \cite[Theorem
1.3.12]{Glaz}, $\fd_R(\Hom_R(R[x],I))\leq
\fd_{R[x]}(\Hom_R(R[x],I))$. And it is clear that $I=\Hom_R(R,I)$
is a direct summand of $\Hom_R(R[x],I)$. Hence, $\fd_R(I)\leq
\wGgldim(R[x])$. Then,
$$\wGgldim(R)=\sup\{\fd(I)\mid \text{I injective  R-module}\}\leq
\wGgldim(R[x]).$$ Secondly, we will prove that $\wGgldim(R[x])\leq
\wGgldim(R)+1$. We may assume that $\wGgldim(R)=n<\infty$.
Otherwise, the result is obvious. Let $I$ be an arbitrary
injective $R[x]$-module. From \cite[Theorem 1.3.16]{Glaz},
$\fd_{R[x]}(I)\leq \fd_R(I)+1$. But $I$ is also an injective
$R$-module since $R[x]$ is a free (then flat) $R$-module. Then,
$\fd_{R[x]}(I)\leq \fd_R(I)+1 \leq \wGgldim(R)+1$. Hence,
$$\wGgldim(R[x])=\sup\{\fd(I)\mid \text{I  injective  R[x]-module}\}\leq \wGgldim(R)+1.$$
Finally, we have to prove that $\wGgldim(R[x])\geq \wGgldim(R)+1$.
From the first part of this proof, we may assume that
$\wGgldim(R)=n<\infty$. Otherwise, the result is obvious. Let $I$
be an injective $R$-module such that $\fd_R(I)=n$ (there exists
since Theorem \ref{main result2}). Then, from Lemma \ref{lemma1},
$\id_R(I^+)=n<\infty$. Therefore, by Lemma \ref{lemma2},
$\id_{R[x]}(I^+)=n+1$. Again by Lemma \ref{lemma1},
$\fd_{R[x]}(I)=n+1$. On the other hand, by Lemma \ref{lemma2},
$\id_{R[x]}(I)=1$. Pick an injective resolution of $I$ over $R[x]$
as follows: $0\rightarrow I \rightarrow I_0\rightarrow
I_1\rightarrow 0$ where $I_0$ and $I_1$ are injective
$R[x]$-modules. Then, $n+1=\fd_{R[x]}(I)\leq
\sup\{\fd_{R[x]}(I_0),\fd_{R[x]}(I_1)-1\}\leq \wGgldim(R[x])$.
Therefore, $\wGgldim(R)+1\leq \wGgldim(R[x])$, as desired. This
finish our proof.
\end{proof}

\begin{remark}
Let $M$ be an $R$-module. Using the definition of the character
$M^+=\Hom_{\mathbb{Z}}(M,\mathbb{Q}/\mathbb{Z})$, we see that the
modulation of $M^+$ over $R[x]$ is the same:
\begin{enumerate}
  \item When we consider $M$ as an $R$-module and then  we consider $M^+$ as an $R[x]$-module.
  \item And when we consider $M$ as an $R[x]$-module (by set $xM=0$) from the beginning.
\end{enumerate}
\end{remark}

Now we are able to give a class of non-coherent rings $R_n$ with
infinite weak global dimensions such that $wGldim(R_n)=n$.
\begin{example}
Consider  $R:=K[X]/(X^2)$ the  local Noetherian non semi-simple
quasi-Frobenius ring and let $S$ be a non-coherent commutative
ring with $\wdim(R)=1$. Set, $T_0=R$ and $T_n=R[X_1,X_2,...,X_n]$
the polynomial ring over $R$. Then,

\begin{enumerate}
  \item $\wdim(T_n\times S)=\infty$,
  \item $\wGgldim(T_n\times S)=n$, and
  \item $T_n\times S$ is not coherent.
\end{enumerate}
\end{example}
\begin{proof}
(1) Follows from the fact that $\wdim(R)=\infty$.\\
(2) Clearly, since $R$ is Noetherian and from \cite[theorem
3.5]{Bennis and Mahdou3}, \cite[Corollary 1.2 and Proposition
2.6]{Bennis and Mahdou2} and \cite[Theorem 12.3.1]{Enocks and
janda} we have $\wGgldim(T_n)=\Ggldim(T_n)=n$ and
$\wGgldim(S)=\wdim(R)=1$. Hence, by Theorem \ref{direct product},
$\wGgldim(T_n\times S)=n$, as desired. \\
(3) Clear since $S$ is non-coherent and this completes the proof.
\end{proof}

\bibliographystyle{amsplain}

\end{document}